\documentclass[12pt]{article}
\usepackage[colorlinks,citecolor=blue,urlcolor=blue]{hyperref}
\usepackage{amsmath,amssymb,amsfonts,amsthm,graphicx}
\usepackage[usenames,dvipsnames]{color}
\newcommand{\C}{{\mathbb C}}

\newcommand{\G}{{\mathcal G}}
\newcommand{\be}{\begin{equation}}
\newcommand{\ee}{\end{equation}} 

\newcommand{\old}[1]{}
\newcommand{\eps}{\epsilon}

\newcommand{\R}{{\mathbb R}}

\newtheorem{theorem}{Theorem}
\newtheorem{thm}[theorem]{Theorem}
\newtheorem{conjecture}{Conjecture}

\newtheorem{corollary}[theorem]{Corollary}

\title{Families of convex tilings}
\author{Richard Kenyon\footnote{Department of Mathematics, Yale University, New Haven; richard.kenyon at yale.edu.}}
\begin{document}

\date{}
\maketitle
\abstract{
We study tilings of polygons $R$ with arbitrary convex polygonal tiles. 
Such tilings come in continuous families obtained by moving tile edges parallel to themselves (keeping edge directions fixed). 
We study how the tile shapes and areas change in these families. In particular we show that if $R$ is convex, the tile shapes can be arbitrarily prescribed
(up to homothety). We also show that the tile areas and tile ``orientations'' determine the tiling.
We associate to a tiling an underlying bipartite planar graph $\G$ and its corresponding
Kasteleyn matrix $K$.
If $\G$ has quadrilateral faces, we show that $K$ is the differential of the map from edge intercepts to tile areas,
and extract some geometric and probabilistic consequences. 
}

\section{Introduction}

In 1903 Dehn \cite{Dehn} showed that an $a\times b$ rectangle can be tiled with squares (of not necessarily the same size) if and only if $a/b$ is rational. Tutte \cite{Tutte} gave an analogous result for tilings
of convex regions with equilateral triangles.
These results were generalized in \cite{Kenyon.dirichlet} where we gave a correspondence between tilings with ``horizontal trapezoids" (trapezoids with two horizontal edges) and harmonic functions on planar Markov chains.
In a further generalization, with Scott Sheffield \cite{KS} we discussed tilings with general convex polygons, relating them
to Kasteleyn theory and the Matrix-Tree Theorem.

In this paper we study \emph{families} of tilings of polygonal regions with convex polygons having fixed edge directions. 
We obtain results regarding the space of possible shapes of tiles and their possible areas. 
We also explore the connections between tilings and the dimer model on the underlying bipartite graph. 

Our first result is illustrated in Figure \ref{tritiling}.
Given a convex tiling of a convex polygon $R$, and for each tile $t$, a new copy $t'$ of $t$ with the same edge directions but of possibly different shape
we show (Theorem \ref{tilingtotiling}) that there is a unique, up to homothety,
``combinatorially equivalent'' tiling of a new convex polygon $R'$, with homothetic copies of the new tiles.
A similar statement holds when $R$ is not convex. 
This result is analogous to a theorem about packings of strictly convex bodies due to Schramm \cite{Schramm}. 
(See also the remarks after Corollary \ref{globalc}.)
\begin{figure}
\begin{center}\includegraphics[width=2.in]{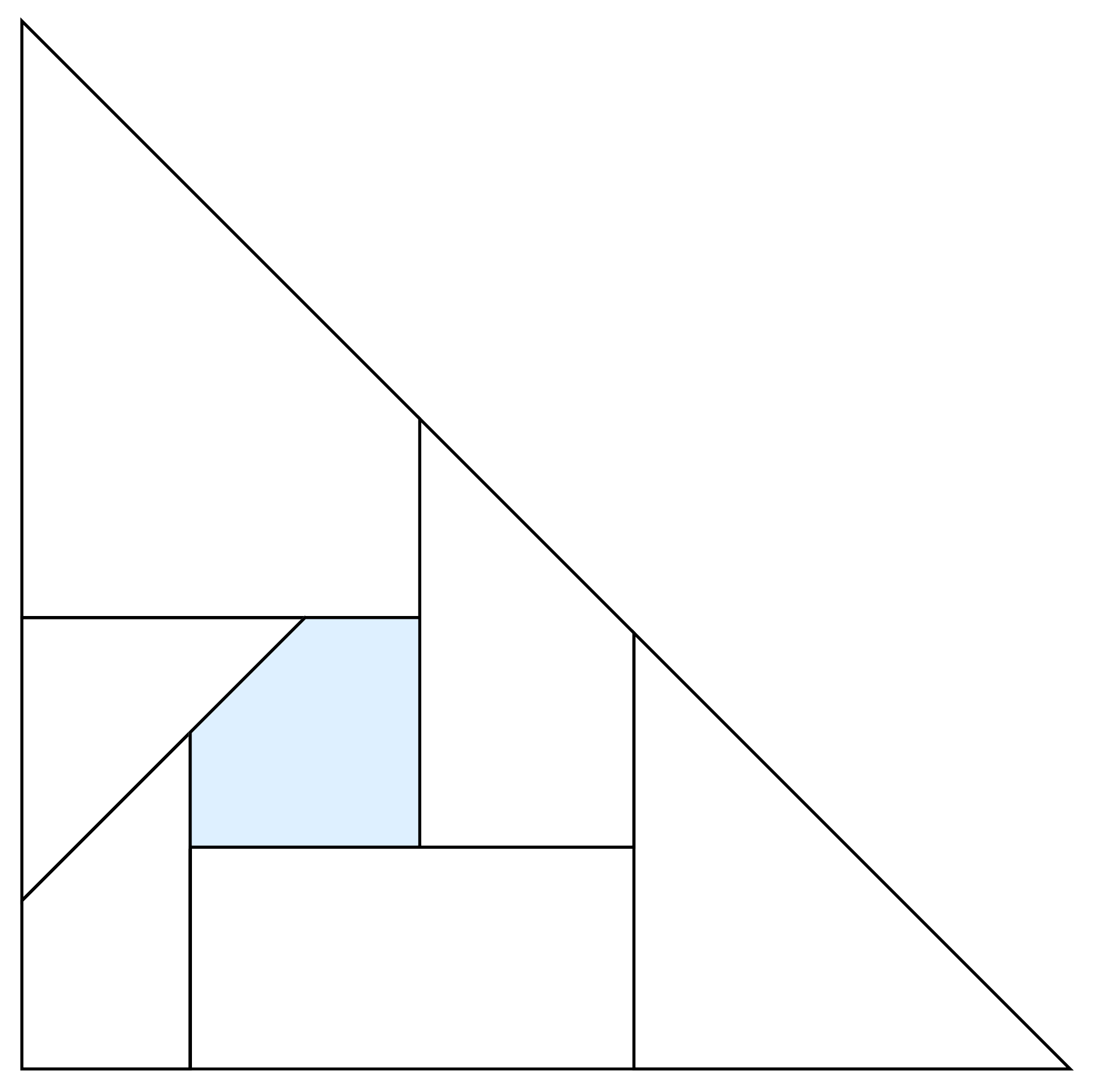}\hskip1cm\includegraphics[width=2.in]{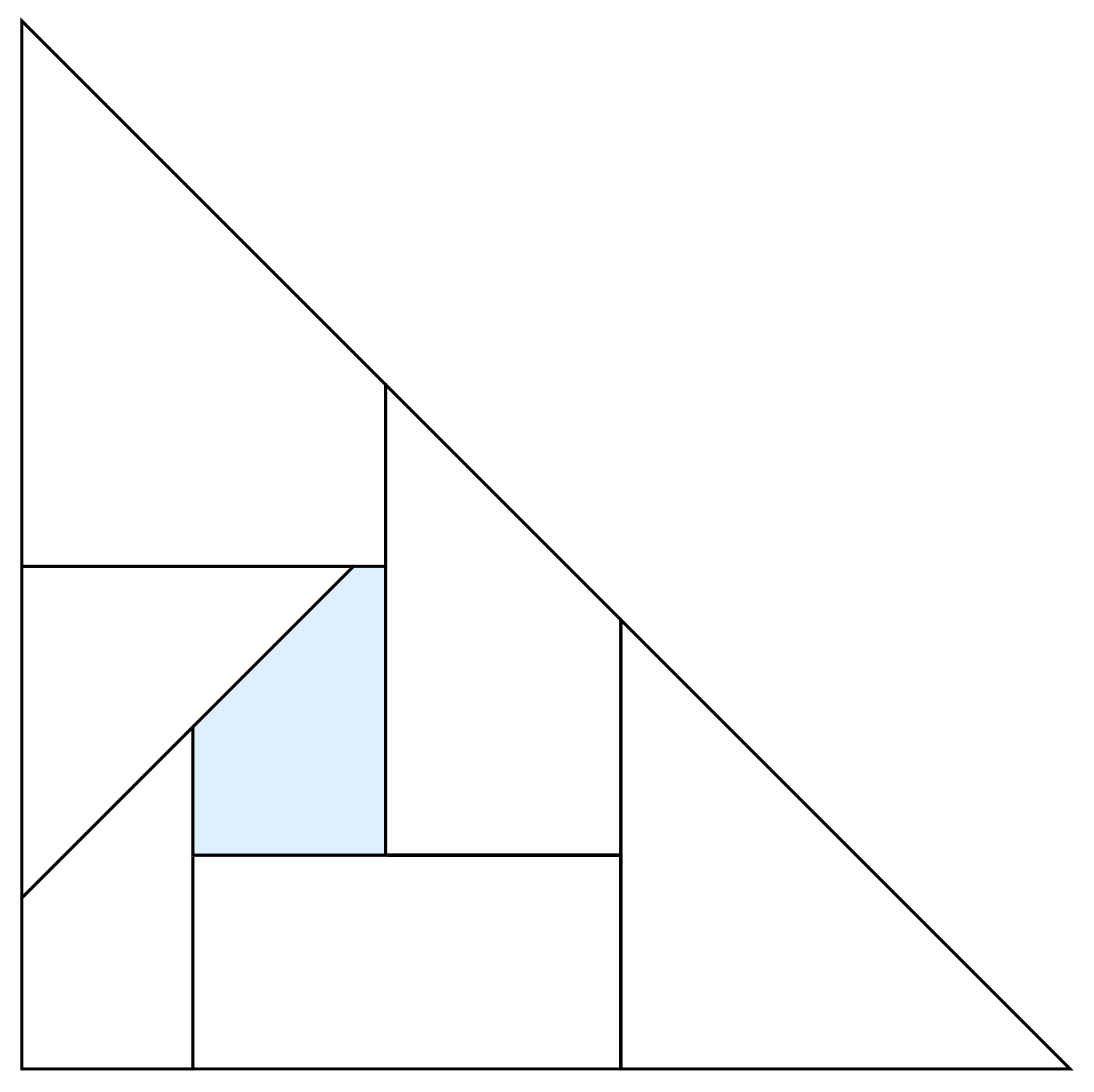}\end{center}
\caption{\label{tritiling}Combinatorially equivalent tilings in which one of the tiles (shaded) has changed shape. The other tiles have changed only by homothety.}
\end{figure}

Another set of results involves the study of the tile areas and orientations for a given tiling family.
(We define ``tile orientation" below---a tile changes orientation if it undergoes a $180^\circ$ rotation in the plane.)
Figure \ref{triareas} shows two combinatorially equivalent tilings with corresponding tiles having the same areas.
Notice that the small triangle has its orientation flipped from one tiling to the next.
\begin{figure}
\begin{center}\includegraphics[width=2.in]{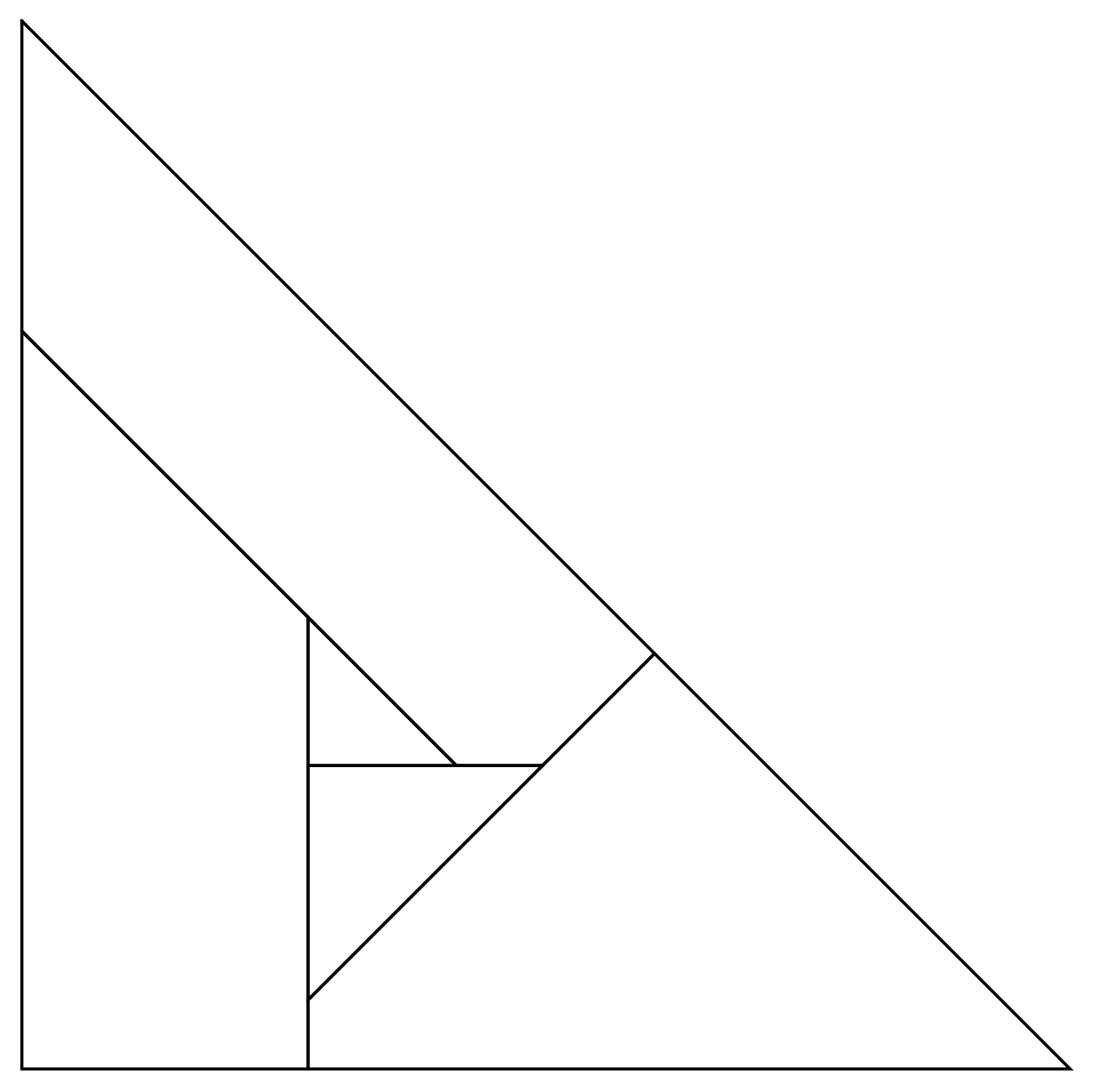}\hskip1cm\includegraphics[width=2.in]{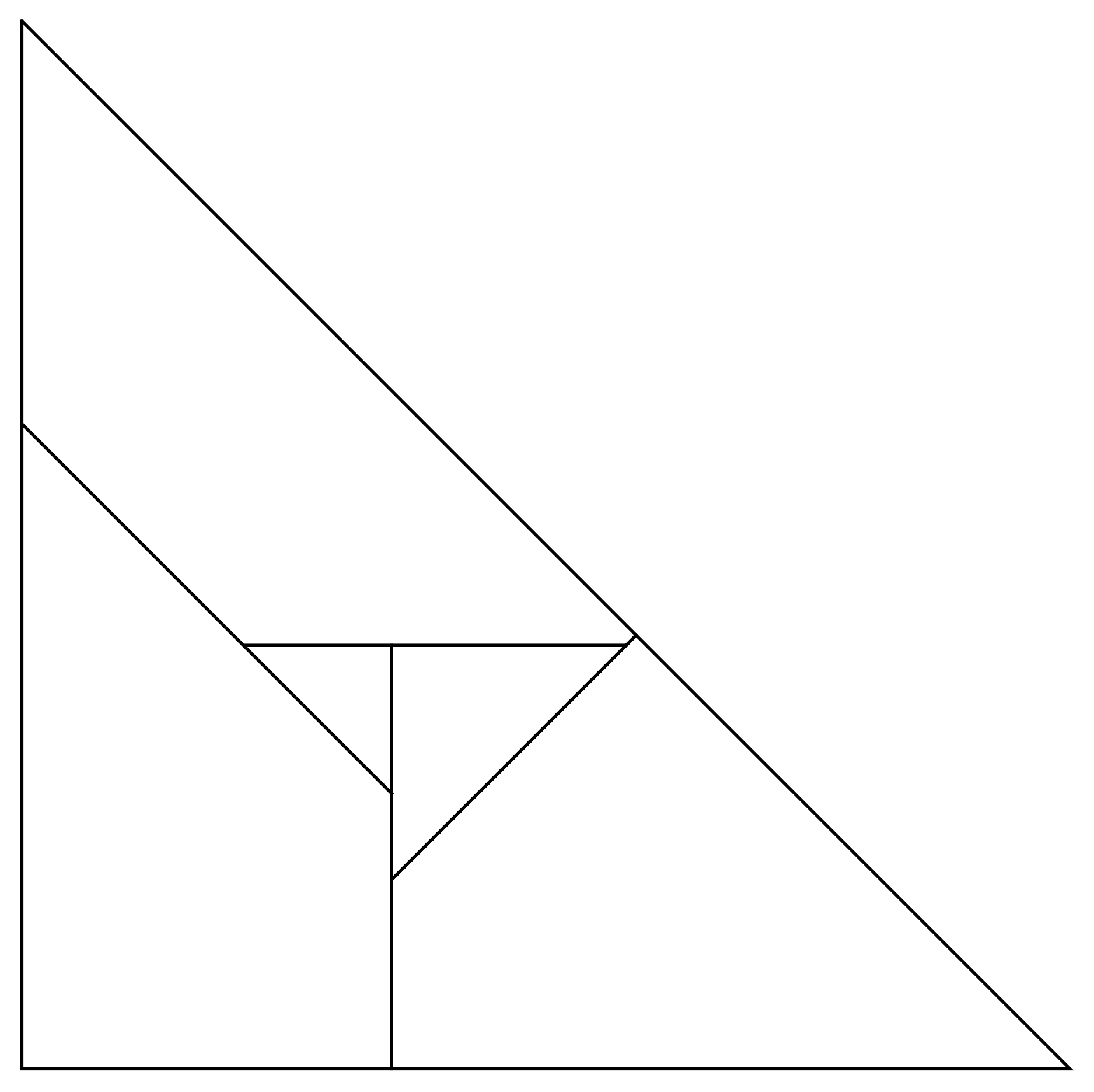}\end{center}
\caption{\label{triareas}Combinatorially equivalent tilings with corresponding tiles of the same area.}
\end{figure}

It is a remarkable fact about \emph{rectangle} tilings of a rectangle 
that one can find for any tiling an equivalent rectangle tiling with 
\emph{any} tuple of positive areas \cite{WKC}. Moreover for prescribed areas the set of equivalent rectangle tilings is in bijection
with a set of st-orientations (acyclic orientations with a unique source and sink) of a related graph, see \cite{AK}.
This statement seems to hold only for rectangle tilings. For other families of convex tilings, like the one illustrated in Figure \ref{triareas}, the situation is more complicated. 
We show (Theorem \ref{areathm}) that for
a given tuple of areas and orientations there is at most one tiling. Furthermore, if the underlying bipartite graph (see below)
has quadrilateral faces, the set of feasible
areas for a given choice of orientations is either empty or topologically a ball of full dimension $d=(\text{\# of tiles})-1$
(Corollary \ref{ballcor}). 

If however we generalize the notion of tiling
to ``homology tiling" (which is defined via a winding number condition), 
we have a clean statement: 
given any tiling of a convex polygon $R$
and a new tuple of areas, summing to the area of $R$,
there is a unique combinatorially equivalent homology tiling of $R$ with the new areas and the same tile orientations. See Theorem \ref{homologytilingthm}.
Since for rectangles any homology tiling is an actual tiling, this generalizes the above.
We furthermore conjecture that if no two interior edge directions are parallel, there is a unique homology tiling of $R$
for any choice of areas and orientations, see Conjecture \ref{2^n} and Figure \ref{all8}.
\bigskip

By a \emph{tiling} of a region $R\subset\R^2$ 
we mean a covering of $R$ with convex polygons, with nonintersecting interiors, as in Figure 
\ref{tritiling}. 
We do not require the tiles to line up edge-to-edge. 
In fact in this figure there are no interior vertices where every tile meets at one of its corners: 
we will consider such a vertex to be 
\emph{degenerate}. We take the point of view that such a vertex (if in the interior) is a degeneration of a tiling with one more tile, located
at that vertex, which has shrunk to a point, as in Figure \ref{degeneracy}. 
\begin{figure}
\begin{center}\includegraphics[width=1.5in]{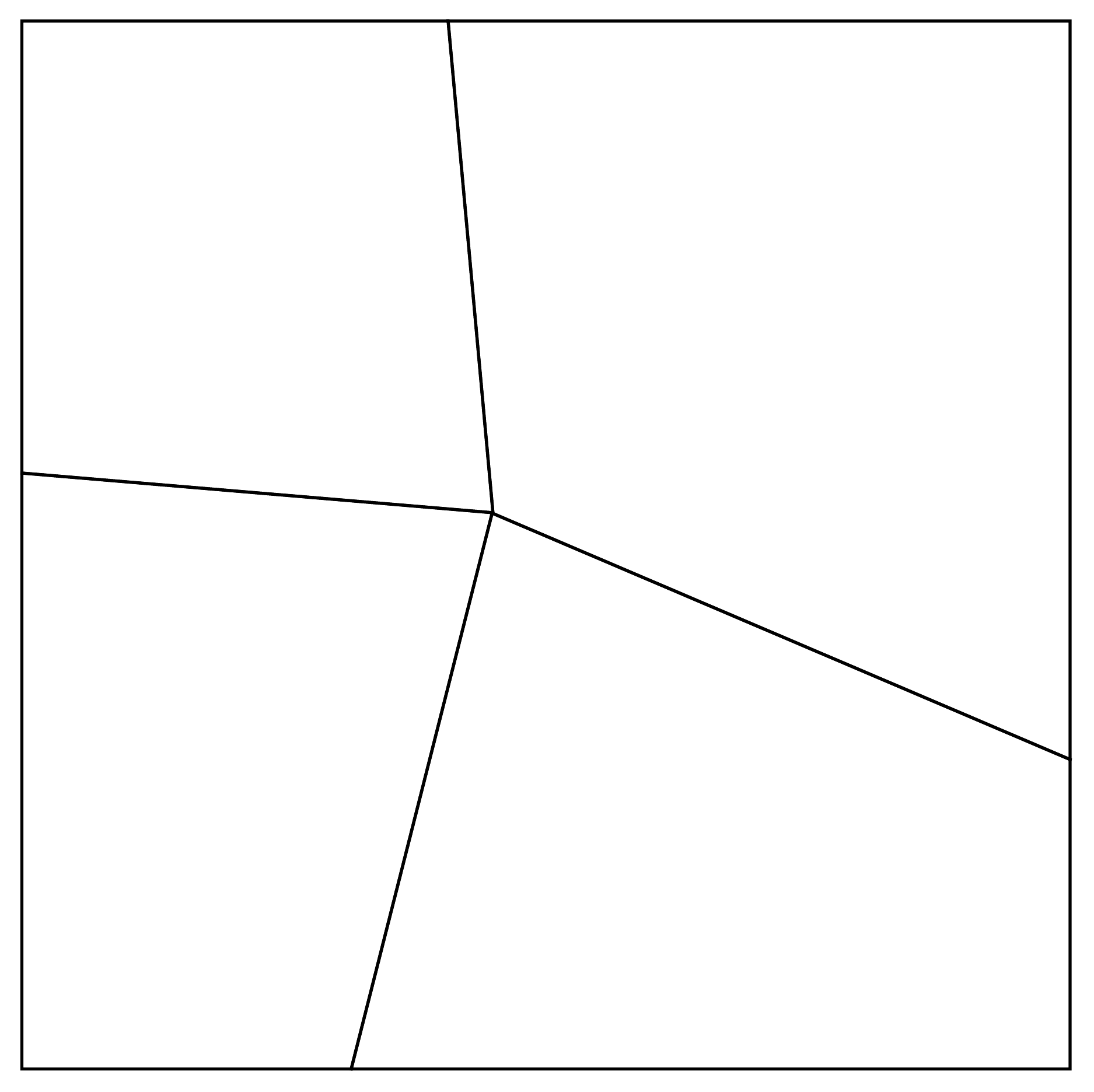}\hskip1cm\includegraphics[width=1.5in]{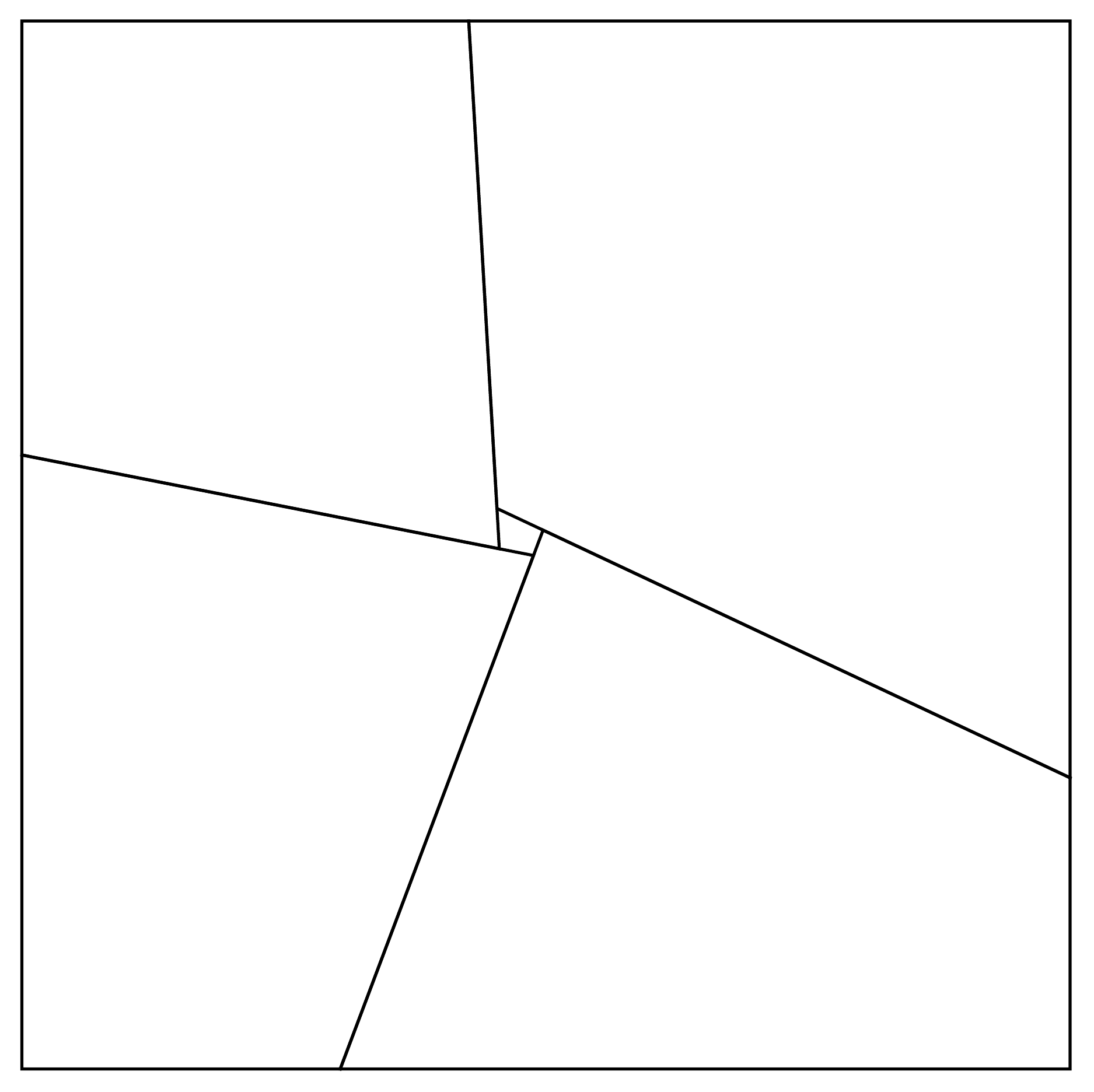}\end{center}
\caption{\label{degeneracy}A convex tiling with a degenerate tile, and a nearby nondegenerate tiling. }
\end{figure}
\begin{figure}
\begin{center}\includegraphics[width=1.5in]{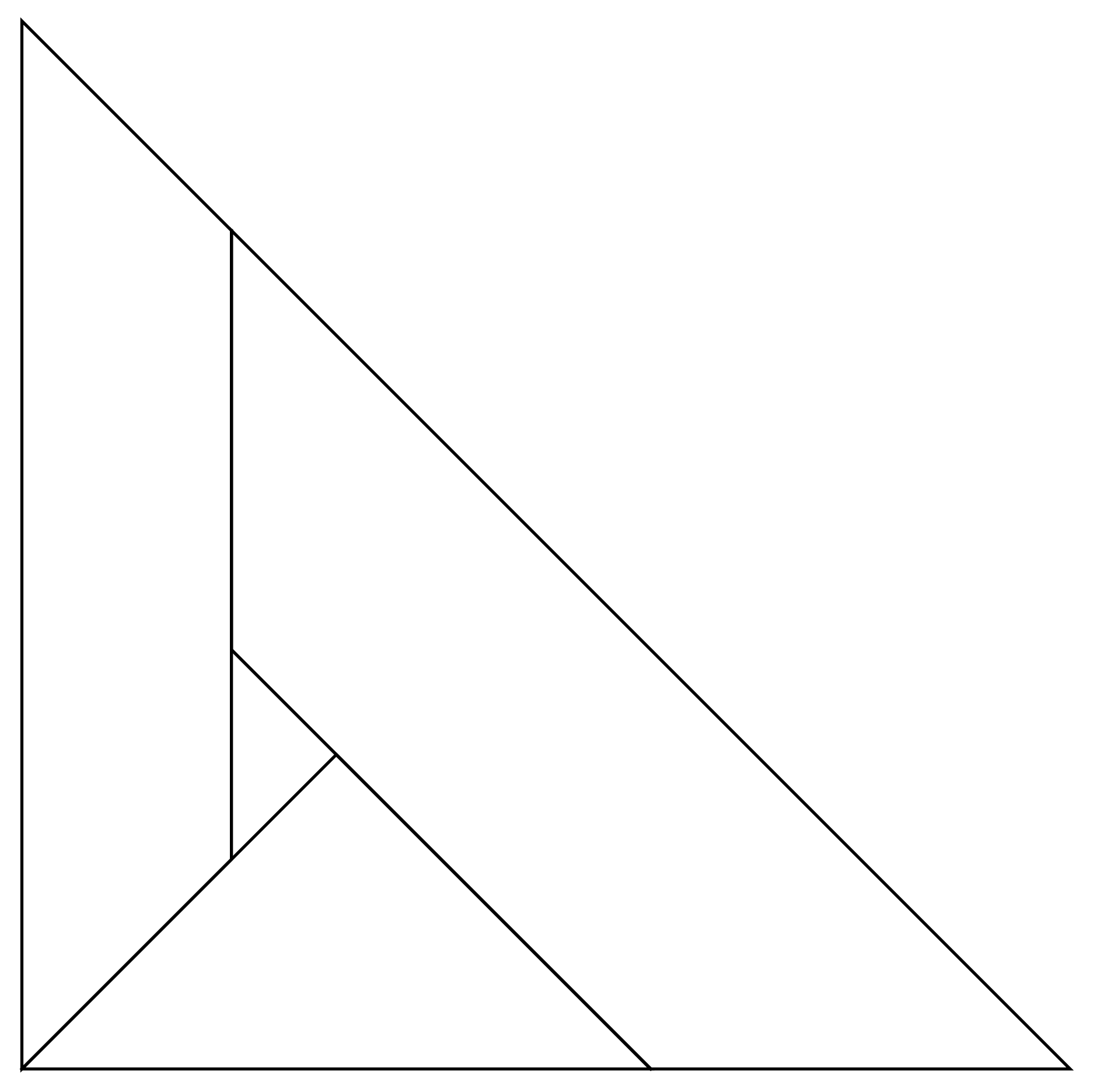}\hskip1cm\includegraphics[width=1.5in]{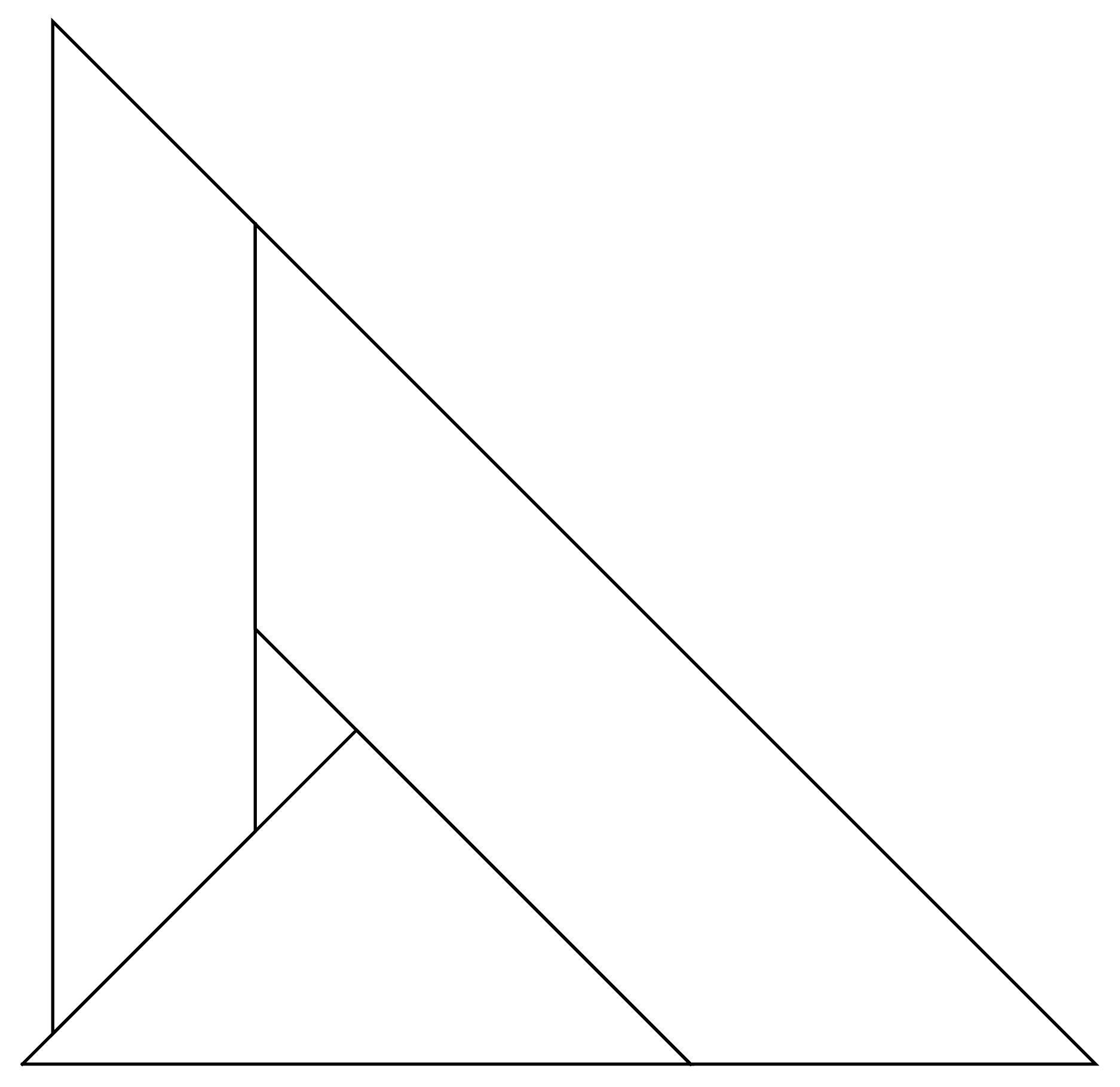}\end{center}
\caption{\label{cornerdeg}A degenerate tiling, and a nearby nondegenerate tiling.}
\end{figure}
Similarly, if multiple tiles meet at a corner of $R$, each tile meeting at one of its corners, then we consider such a tiling to be a degeneration of a tiling of a polygon $R'$ having an extra edge or edges
at that corner, whose lengths have gone to zero, as in Figure \ref{cornerdeg}. 
This point of view on degeneracies is slightly unusual,
but as we vary tilings in families such degeneracies are natural. In a nondegenerate tiling,
each vertex of a tile is at a `T' (or a T with several legs, as for the lower middle vertex in the right panel of Figure \ref{Kmatproperty}),
except for the convex corners of $R$. These tilings are called ``t-graphs'' in \cite{KS}. 

Associated to a (nondegenerate) tiling is a bipartite planar graph $\G$. It has a white vertex for each tile, a black vertex
for each maximal line segment in the tile boundaries, and an edge for an adjacency. See Figure \ref{Bgraph}.  
\begin{figure}
\begin{center}\includegraphics[width=2.in]{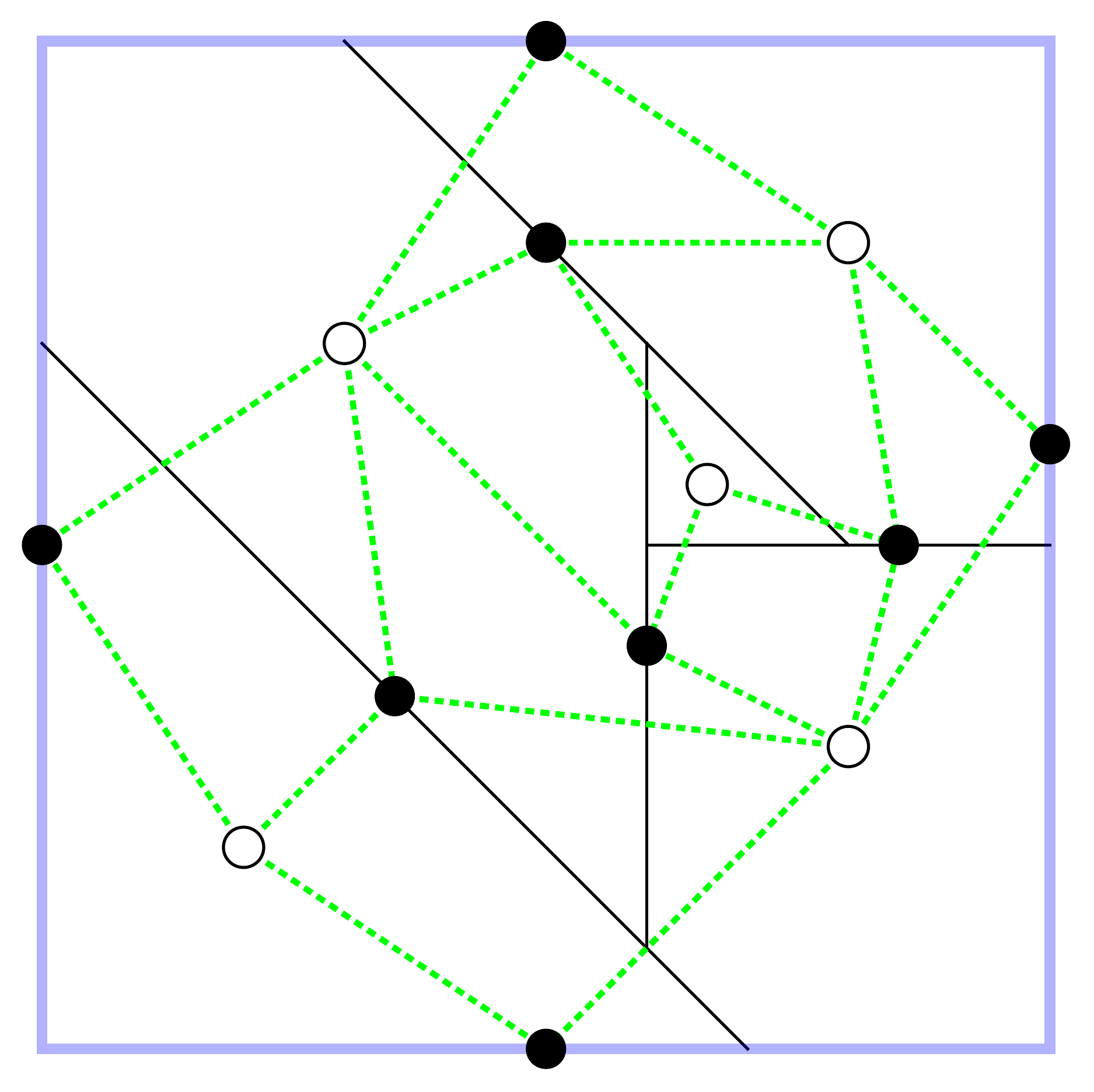}\end{center}
\caption{\label{Bgraph}Bipartite graph (green) associated to a convex tiling. Two tiling of $R$ 
are combinatorially equivalent if they have the same underlying bipartite graph
with the same boundary identifications.}
\end{figure}
Note that some black vertices come from segments along the 
boundary of the tiled region; we'll call these \emph{boundary vertices}.
Each quad face of $\G$ corresponds to a ``T'' in the tiling: two segments meet, one of which end in the interior of another.
Bounded faces of higher degree in $\G$ correspond to ``T''s with multiple legs: two or more segments end at an interior point of another segment.

We call two tilings \emph{combinatorially equivalent} if their graphs are isomorphic respecting boundary vertices,
that is, there is an isomorphism fixing the boundary vertices pointwise, by which we mean fixing which edge of the boundary polygon they correspond to.

To encode the tiling completely (up to translation), we can record, for each edge $e=wb$ of the bipartite graph $\G$, the edge vector in $\C$ giving the
counterclockwise boundary edge of tile $w$ as a subset of segment $b$. This turns $\G$ into a bipartite \emph{network}, that is,
an edge-weighted bipartite graph, with complex weights.


As we vary the shapes of tiles in the above families, 
we can record the areas as functions of the edge intercepts.
If bounded faces of $\G$ are quadrilateral,
the map $\Psi$ from edge intercepts to tile areas has differential  $D\Psi=K$,
where $K$ is a \emph{Kasteleyn matrix}, a matrix whose maximal minors counts perfect matchings of the underlying graph $\G$. See Section \ref{KasD}.
This observation, along with the reconstruction of a tiling from a graph given in \cite{KS} (Theorem \ref{convexgraphtotiling} below) 
allows us to realize the Kasteleyn matrix of a bipartite planar graph (with quad faces), up to gauge equivalence,
as the differential of a mapping. 
Some geometric and probabilistic implications of this fact are given in Section \ref{KasD} and 
Theorem \ref{diffthm}.

\bigskip

\noindent{\bf Acknowledgments.} This research was supported by NSF grant DMS-1940932 and the Simons Foundation grant 327929.
We thank Sebastien Franco, Gregg Musiker, David Speyer and Lauren Williams for comments and
discussions on this project.

\section{Background on dimers}

For more background on the dimer model see \cite{Kenyon.dimers}.
A dimer cover, or perfect matching, of a graph $\G$ is a collection of edges with the property that each
vertex is the endpoint of a unique edge in the collection. 
If $\nu= \{\nu_e\}_{e\in E}$ is a positive edge weight function on $\G$, we can associate to each dimer cover $m$ a weight
$\nu(m) = \prod_{e\in m}\nu_e$. 

Kasteleyn \cite{Kast67} showed how to compute the weighted sum of dimer covers of any planar graph.
When $\G$ is bipartite and planar, this sum is the determinant of a \emph{Kasteleyn matrix} for $\G$,
defined as follows. 
Let $K=(K_{w,b})_{w\in W,b \in B}$ be a matrix with rows indexing white vertices and columns indexing black vertices and
$$K_{w,b} = \begin{cases} c_{wb}\nu_{wb} &\text{if $w\sim b$}\\0&\text{else}\end{cases}$$
where $c_{wb}$ is a complex number of modulus $1$ with the property that the alternating product of $c$'s around
a bounded face of length $2\ell$ is $(-1)^{\ell+1}$. Here by alternating product around a face 
$w_1,b_1,w_2,\dots,w_\ell,b_{\ell}$ we mean the first, divided by the second, times the third, and so on: 
$$\frac{c_{w_1b_1}c_{w_2b_2}\dots c_{w_{\ell},b_{\ell}}}{c_{w_2b_1}\dots c_{w_1b_{\ell}}} = (-1)^{\ell+1}.$$

\begin{thm}[Kasteleyn \cite{Kast67}]\label{Kastthm} Let $K$ be a Kasteleyn matrix for a planar bipartite network $\G$. Then
$$Z:=\sum_{\text{dimer covers $m$}}\nu(m) = |\det K|.$$
\end{thm}

Any two choices $K,K'$ of Kasteleyn matrix for $\G$ are diagonally equivalent:
$K' = D_WKD_B$ where $D_B,D_W$ are diagonal matrices with diagonal entries of modulus $1$. 
In Kasteleyn's original formulation the complex signs were just $\pm1$; however it is easy to see that
any Kasteleyn matrix with complex signs is diagonally equivalent to one with real signs. 

Given two edge weight functions $\nu,\nu'$ on $\G$ we say they are \emph{gauge equivalent} if there is a (nonzero) function $f$ on vertices 
such that
$\nu'_{wb} = \nu_{wb}f(w)f(b)$. 
Gauge equivalent weights give rise to the same probability measure on dimer covers, 
where a dimer cover has probability proportional to its weight (under a gauge equivalence, the constant of
proportionality will change but not the probability).

Edge weight functions $\nu,\nu'$ are gauge equivalent if their Kasteleyn matrices are diagonally equivalent (with diagonal matrices $D_B, D_W$ with diagonal entries $f(b), f(w)$ respectively).
Edge weight functions $\nu,\nu'$ are gauge equivalent if and only if they have the same \emph{face weights}, where the face weight $X_f$ of a (bounded) face $f$ with vertices $w_1,b_1,\dots,b_\ell$ is 
$$X_f = \frac{\nu_{w_1b_1}\nu_{w_2b_2}\dots \nu_{w_{\ell},b_{\ell}}}{\nu_{w_2b_1}\dots \nu_{w_1b_{\ell}}}.$$

Given an arbitrary set of positive face weights $X_f$, it is easy to construct a corresponding edge weight function
giving rise to those face weights. So the set of
equivalence classes of edge weights is globally parameterized by $(\R_+)^{F}$, where $F$ is the set of bounded faces of $\G$. 

We note that the gauge transformations for the edge weights and edge signs are essentially the same operation:
multiplying on the left and right by either a positive real diagonal matrix or a diagonal matrix of complex numbers of modulus $1$. Therefore we can combine these two operations into a single one.
We say two Kasteleyn matrices $K,K'$ are \emph{gauge equivalent} if they are diagonally equivalent via
(nonsingular) complex-valued diagonal matrices. Gauge equivalent Kasteleyn matrices correspond to the same 
probability measure on dimer covers. 

\old{Finally, if $\G$ has some bounded faces which are not quads, we can consider $\G$ to be a limit of a graph $\G_\eps$ with quadrilateral faces in which some edge weights have gone to zero. We construct $\G_\eps$ by quadrangulating each non-quad face of 
$\G$ in such a way as to preserve the bipartite structure, and each new edge is given weight $\eps$. As $\eps\to0$ we recover $\G$. This device is only needed in Sections \ref{KasD} and \ref{Homtilings}.}

\section{All shapes are possible}

\subsection{Space of shapes of a tile}
Given a convex polygon $t$, let $C_t$ be the space of convex polygons with the same number of edges and the same edge directions (in the same cyclic order), up to a global sign, and up to translation.
The space $C_t$ naturally has a linear structure: if $t$ has edges $z_1,\dots,z_n$ in counterclockwise order
then any other tile in $C_t$ will have edges $a_1z_1,\dots,a_nz_n$ for reals $a_i$, all of the same sign: all positive or all negative,
subject to the linear condition $\sum a_iz_i=0$. If we parameterize $C_t$ with the $n$-tuple $(a_1,\dots,a_n)$, 
it is naturally the interior of a cone
in an $n-2$-dimensional real vector space $V_t$. More precisely, $C_t$ is the union of two open convex cones which are negatives of each other:
$C^+_t$ consist of those tiles in which $a_i> 0$, and $C^-_t$ those in which $a_i<0$.
These two cones are called the two \emph{orientations} of $t$. 
It is convenient to add the origin to $C_t$ since as we vary tiles in families we sometimes
have to deal with the case where a tile shrinks to a point. 


The vector space $V_t$ in which $C_t$ sits has another natural linear coordinate system, using the intercepts of the lines 
defining the edges of $t$. See Section \ref{KasD} below.

\subsection{Space of shapes of a tiling}

Given a convex tiling $T=\cup_{w\in W} t_w$ of a polygon $R$, let us consider the set $C_T=C_T(R)$ of combinatorially equivalent tilings,
in which each tile has the 
same edge directions, up to sign, as the corresponding tile of $T$. Note that $C_T$ also has a linear structure:
an element of $C_T$ is determined by the shapes of its individual tiles, so 
$C_T$ lies inside the product $\prod_w C_{t_w}\subset\prod_w V_{t_w}$, and is defined by the linear conditions that 
certain sums of edge vectors are zero:
around each interior black vertex of $\G$ there is one such condition,
and for each boundary edge the sum of tile edge vectors equals the vector of the corresponding edge of $R$.

The space $C_T$ has a natural subdivision into regions $C_T(\sigma)$ where the individual tiles have fixed orientations:
$\sigma\in\{-1,1\}^k$ corresponds to a choice of orientation for each tile. 
Each subset $C_T(\sigma)$ is either empty or a convex polytope: given any two tilings $T_1,T_2\in C_T(\sigma)$
then any convex combination of $T_1$ and $T_2$ is also a tiling of $C_T(\sigma)$.
The full space $C_T$ is not typically convex, however.
These polytopes $C_T(\sigma)$
are glued in $C_T$ along parts where one or more tiles degenerate to points.

\old{In particular to $C_T$ there is an associated simplicial complex whose vertices are the feasible orientations $\sigma\in\{1,-1\}^T$,
and whose $k$-dimensional faces are sets of $k$ tiles which can simultaneously degenerate to points. }

\subsection{Kasteleyn matrix}

Let $T=\cup_{w\in W} t_w$ be a tiling of a convex polygon $R$, and let $\G$ be
the associated bipartite network. 
We note that $\G$ has $n-1$ more black vertices than white vertices, 
where $n$ is the number of edges of $R$. This fact follows from
an Euler characteristic argument: the tiling divides $R$ into $|W|$ open $2$-cells, $|B|$ open line segments and $n$ vertices.
By the Euler characteristic,
$$\chi(R) = 1 = |W| - |B| + n.$$
We let $B_{int}$ be the interior black vertices, and $B_{\partial}$ the boundary vertices.
Thus $|W|=|B_{int}|+1$. 

Let $K$ be the weighted bipartite adjacency matrix of $\G$;
$K$ has rows indexing the white vertices (tiles of $T$) and columns indexing the black vertices (maximal segments in the tile boundaries).
The entry $K_{w,b}$ is the complex number giving the vector of the edge of tile $w$ in direction $b$, 
oriented counterclockwise around the tile
boundary. 
Note that $K$ is a \emph{Kasteleyn matrix} for $\G$, in the sense that the alternating product of edge weights around any face of $\G$ 
is real with sign $(-1)^{\ell+1}$ where $2\ell$ is the degree of the face (the number of edges on its boundary).
See Figure \ref{Kmatproperty} for the proof of this fact.
\begin{figure}
\begin{center}\includegraphics[width=2.in]{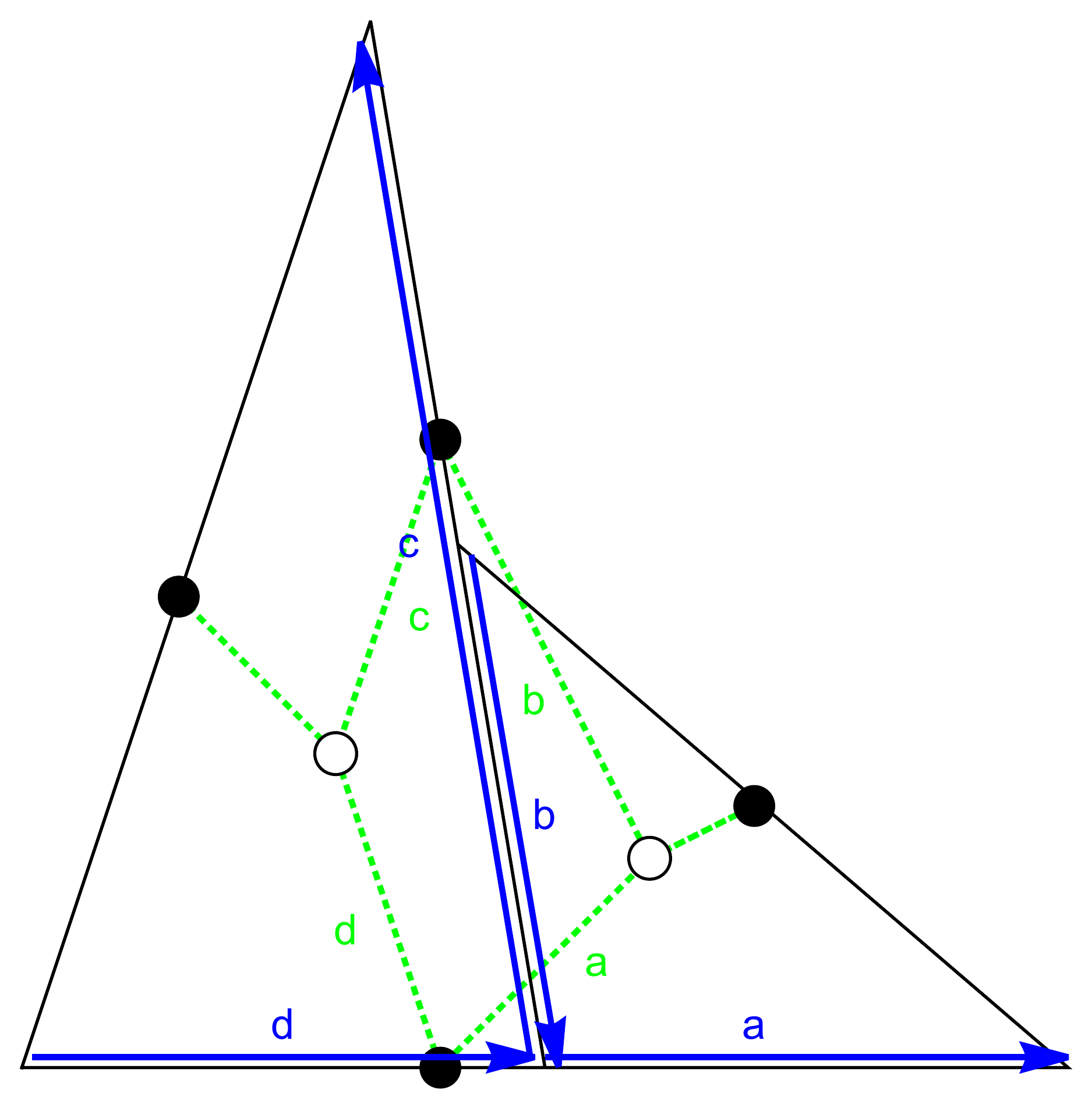}\hskip1cm\includegraphics[width=2.in]{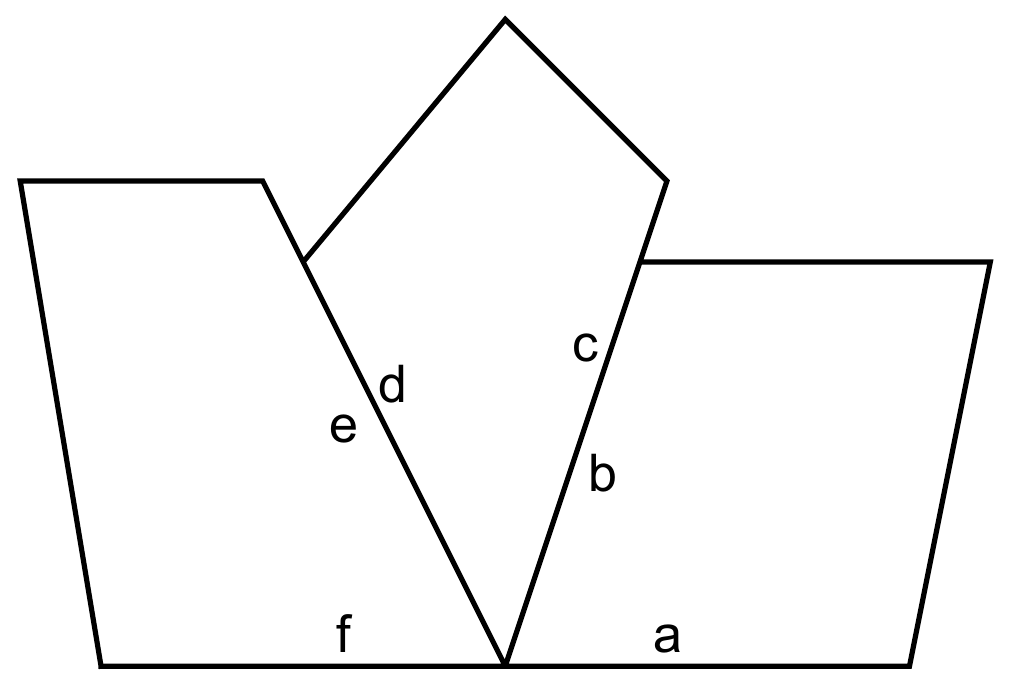}\end{center}
\caption{\label{Kmatproperty}In the first panel, the lower central vertex in blue corresponds to the green quadrilateral face of $\G$, with edge weights $a,b,c,d$. We have 
$\frac{ac}{bd}<0$ since $a,d$ point in the same direction and $b,c$ in opposite directions.
Similarly, in the second panel the lower central vertex corresponds to a hexagonal face; we have $\frac{ace}{bdf}>0$. }
\end{figure}

It is a bit harder to see that $K$ has full rank $|W|$, but this is important. 
That $K$ has full rank follows from Kasteleyn's theorem, Theorem \ref{Kastthm}.
We simply need to show that, once we remove all but one of the black boundary vertices from $\G$,
the new graph $\G'$ has at least one dimer cover. This is a surprising fact since it relates a geometric property (of being a convex tiling) 
to a combinatorial one. We can prove this is as follows.
Rotate $R$ so that one boundary edge $b_0$ is horizontal.
Rotate further by a very small positive amount that so that each tile $w$ has a unique vertex $v_w$ with lowest $y$-value. 
Since $v_w$ is a `$T$', exactly one of the two edges of face $w$ meeting at $v_w$ has $v_w$ as its lower endpoint; match $w$ to the
black vertex corresponding to this edge. The leftmost face containing $b_0$ along one of its edges is matched to $b_0$.
This concludes the proof that $K$ has maximal rank $|W|$.

\subsection{From tiling to tiling}

\begin{thm}\label{tilingtotiling}
Given a tiling $T=\cup t_w$ of a convex polygon $R$ and, for each tile $t_w$, a new tile $t'_w$ in $C_{t_w}$, 
there is a unique (up to homothety) combinatorially equivalent tiling $S=\cup s_w$ of a convex polygon 
$R'\in C_R$,
where $s_w=a_wt'_w$ is a homothetic copy of
$t_w'$, and where each $a_w\in\R$ could be positive, zero or negative.
\end{thm}

\begin{proof} In order to make a tiling out of the tiles $a_wt'_w$, the reals $\{a_w\}$ must satisfy a linear relation around each interior segment, that is, at each non-boundary black vertex $b\in B_{int}$.
More precisely, let $K$ be the Kasteleyn matrix associated to $T$, and let $K'$ be the matrix
obtained from $K$ by replacing the row for $t_w$ by the corresponding row for $t'_w$
(that is, the row with the same nonzero entries, but these corresponding to the edge vectors of $t'_w$ instead of $t_w$). 
The equation on the $\{a_w\}$ is
$\sum_w a_w K'_{wb}  =0$ for non-boundary vertices $b$.
Since $K'$ still has full rank $|W|$, and $|B_{int}|=|W|-1$, there is a unique solution vector $(a_w)_{w\in W}$ 
up to scale.
Since each column of $K'$ has a constant argument, the unique solution can be chosen real. 

It remains to see that this solution defines a tiling. Let $\G^*$ be the dual graph of $\G$, without the external vertex. 
We glue in a disk for each bounded face of $\G^*$ (each non-boundary vertex of $\G$) to make a topological disk $D$ in which $\G^*$ is embedded. 
The initial tiling defines a continuous map from $D$ to $R'$ which collapses each black face of $\G^*$
to a segment, and maps each white face 
to the corresponding tile in an orientation-preserving fashion.
We can deform the tiles $t_w$ continuously from their initial values (since the space of shapes $\bar C_t$ of each tile is connected);
along the deformation the map from $D$ remains orientation-preserving, sending the boundary of $D$ homeomorphically to the boundary of $R'$ and preserving the boundary orientation.  The orientation of the mapping is preserved even 
if the individual tiles
change their tile-orientation (which rotates a tile by $\pi$, but preserves its orientation as a map), and the mapping is surjective 
to $\R'$, so defines a tiling.
\end{proof}

If $R$ is initially not convex, the same argument works, but since each boundary edge of $R$ will have its length adjusted to make $R'$, this might result in $R'$ being self-intersecting; nonetheless there is at least a locally injective tiling
of $R'$.

\subsection{From a graph to a tiling}

Suppose we are just given the planar bipartite graph $\G$ with positive edge weights $\{\nu_{wb}\}$, and a convex polygon $R$, along with a cyclic identification of the black boundary vertices of $\G$ with the edges of $R$. 
Then under the certain nondegeneracy condition there is an associated tiling. 
We say $\G$ is \emph{$2$-nondegenerate} if
every pair of edges having distinct endpoints (at most one of the four of which is a boundary vertex) can be completed to a dimer cover of $\G$
using exactly one of the boundary vertices.
For example if $\G$ has a degree-$2$ vertex it is \emph{not} $2$-nondegenerate. See Figure \ref{2deg} for another $2$-degenerate example.
\begin{figure}
\begin{center}\includegraphics[width=2.in]{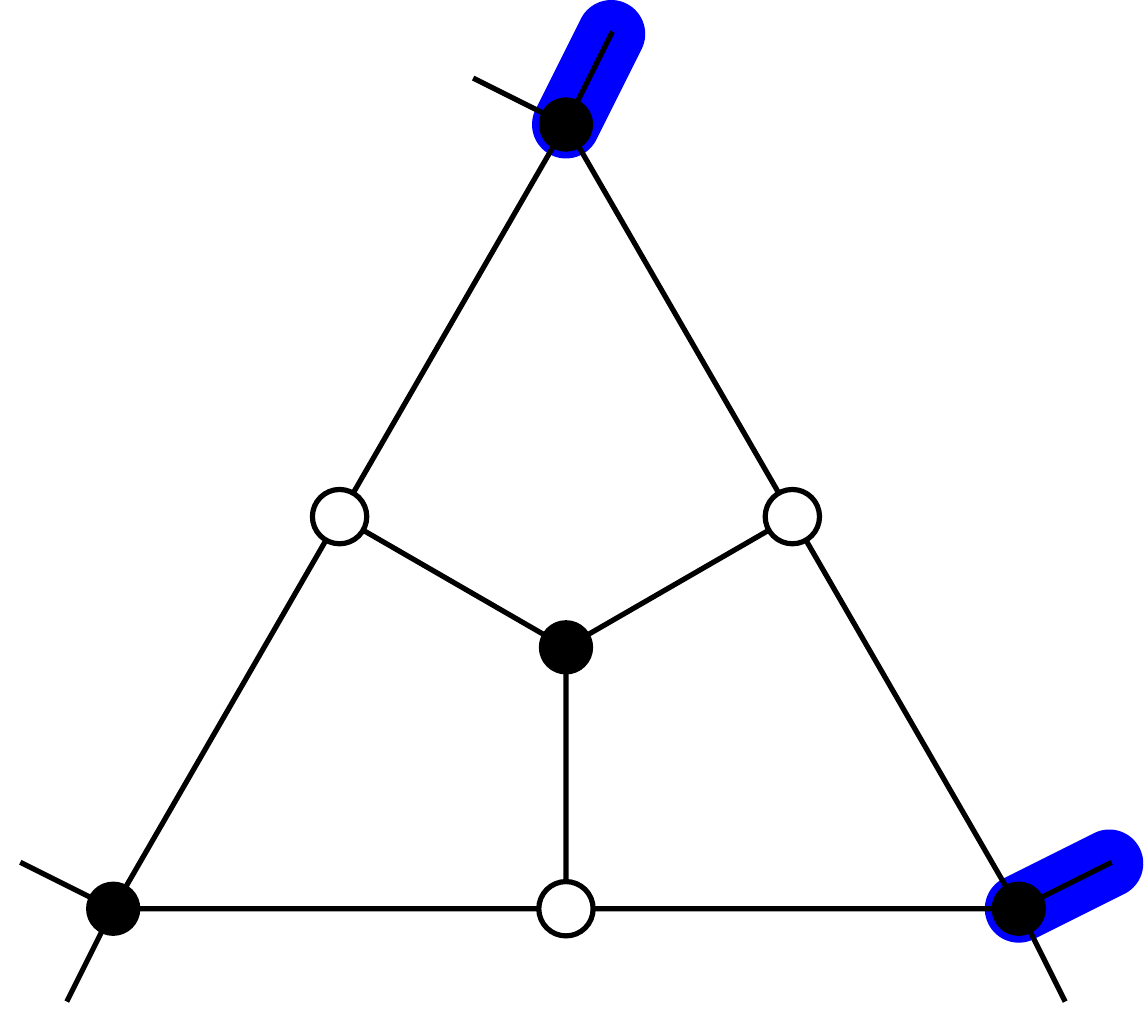}\end{center}
\caption{\label{2deg} A graph which contains a subgraph as shown (which is connected to the rest of the graph only via the three black corner vertices) is $2$-degenerate; for example no dimer cover
contains both blue edges, since the three black vertices have among them only two remaining neighbors.
In any tiling with this graph the tiles for the three white vertices will be reduced to segments. If we swap the colors $W\leftrightarrow B$, however,
then a nondegenerate tiling may exist, but in any such tiling the tiles for the four white vertices will have convex union.}
\end{figure} 
The author and Sheffield proved the existence of a tiling under these conditions: 

\begin{thm}[\cite{KS}]\label{convexgraphtotiling} Suppose $\G,R$ are as above and $\G$ is $2$-nondegenerate.
Then there is a unique gauge equivalent choice of edge weights and a unique tiling of $R$ with those edge weights. 
\end{thm}

\begin{proof} 
See \cite{KS} for the complete proof. We'll give the construction (and prove unicity) here.
Let $q_1,\dots,q_n\in\C$ be the counterclockwise boundary edges of $R$.
Let $K$ be a
Kasteleyn matrix for $\G$. We need to find diagonal matrices $D_W,D_B$ so that 
$K' = D_WKD_B$ defines the tiling, that is, $K'1_B = 0$ and $1_WK'$ is the vector which is zero on internal
black vertices and $q_i$ on the $i$th boundary black vertex. (Here $1_B$ and $1_W$ are the ``all-$1$s'' vectors.)
The solution is given in three linear algebraic steps. 

First, find a nonzero vector $v_W$ so that $v_WK$ is zero on $B_{int}$.
Since $K$ restricted to $B_{int}$ has full rank and has one more row than column, there is a unique
solution up to scale. Define $D_W$ to be the diagonal matrix with entries of $v_W$.

Second, for $b\in B_{\partial}$, the sum of the $b$th column of $D_WKD_B$
needs to add to $q_i$; this determines $(D_B)_{b,b}$. 

Finally, the remaining $|B_{int}|$ entries of $D_B$ are determined by the linear equations $KD_B1_B=0$. (Here there
is one more equation than variable but since the sum of each column is zero the last equation is a consequence
of the previous ones).

The fact that $K'$ determines an actual tiling is a consequence of the maximum principle: we refer the reader to \cite{KS}. 
The $2$-nondegeneracy implies that white tiles do not degenerate to segments.
\end{proof}

Since the edge weights modulo gauge are described by the face weights $\{X_f\}$, we have the following.

\begin{corollary}\label{globalc} For fixed convex $R$ and $2$-nondegenerate $\G$, the space of $(\G,R)$-tilings is homeomorphic to $\R_+^F$, with global coordinates provided by the face weights.
\end{corollary}

A stronger variant of Theorem \ref{convexgraphtotiling} above is as follows. Let $\G$ be a bipartite graph satisfying the above conditions,
and for each white vertex $w$ choose a convex $k$-gon $t_w$ (where $k$ is the degree of $w$), under the constraint
that if two white vertices have a black neighbor in common then the corresponding edges (corresponding to the adjacencies with the black vertex) have the same direction up to sign. Suppose moreover that the directions of the boundary black vertices
can be oriented so that as one moves cclw around the boundary of $\G$, they turn once around, that is, there is a convex
polygon with those edge directions in cyclic order.
Then there is a tiling of a convex polygon with the combinatorics of $\G$ in which for each white vertex $w$ the
corresponding tile is homothetic to $t_w$. This is a tiling analog of the theorem of Schramm \cite{Schramm}.

\section{Areas and the Kasteleyn matrix}

\subsection{The area form} 

Recall the cone $C_t\subset V_t$ of convex polygons with edge directions 
in the directions of a convex polygon $t$. On $C_t$ there is a quadratic form $q_t$ giving the area of the polygon.
It is a homogeneous quadratic function of the edge lengths, defined on all of $V_t$. By a Theorem of Thurston \cite{Thurston}
it has signature $(1,n-3)$ when $t$ is a (convex) $n$-gon.  In particular if we consider the subset of $C_t$ of polygons of fixed area $A>0$,
it consists of two disjoint hyperboloid components $H_+$ and $H_-$, one for each of the two different orientations of $t$.

\begin{thm}\label{areathm} Let $T$ be a nondegenerate convex tiling of a convex polygon $R$ and consider the space of combinatorially 
equivalent tilings of $R$ with the same edge directions and same tile orientations.
Then for a given set of tile areas there is at most one tiling. 
\end{thm}

\begin{proof}
Suppose there are two combinatorially equivalent tilings $T_1,T_2$ of $R$ having the same tile orientations and the same tile areas.
Recall that $C_T(\sigma)$ is convex. 
Consider the tiling $sT_1+(1-s)T_2$ where $T_1,T_2\in C_T(\sigma)$ and $s\in[0,1]$.
If tiles $t_1,t_2\in C_t$ have the same area and orientation then a convex combination $st_1+(1-s)t_2$ has area which is a strictly concave function of $s$ (the line segment joining $t_1$ and $t_2$ lies on one side of the hyperboloid component $H$). In particular the area for $s\in(0,1)$ is larger than the area at the endpoints.
It is impossible that all tile areas increase, since the boundary is fixed.
Therefore the areas of corresponding tiles of $T_1$ and $T_2$ cannot be all equal, unless $T_1=T_2$. Thus the areas are uniquely determined by $R$ and $\sigma$.   
\end{proof}

\subsection{The Kasteleyn matrix as differential}\label{KasD}

For a line $ax+by=c$ let us define the \emph{intercept} to be the signed distance 
from the line to the origin: $\pm c/\sqrt{a^2+b^2}$, with the sign depending on which side the origin is on
(chosen arbitrarily).

Let $T$ be a tiling of a convex polygon $R$, whose underlying graph has quad faces.
In this case the space $C_T$ can be uniquely parameterized by the intercepts 
$\{i_b\}$ for $b\in B_{int}$, that is, $C_T\subset\R^{B_{int}}.$

Let us consider a slightly larger space: fix the outer polygon $R$ except for one edge $b_0$ of $R$, whose intercept we allow to vary. 
Let $B_0 = B_{int}\cup \{b_0\}$. As noted earlier $|B_0|=|W|$. We can redefine $C_T\subset\R^{B_0}$ to be this larger space.
Let $C_T(\sigma)\subset C_T$ be the corresponding subset with fixed tile orientations.

Let $\Psi:C_T\to\R^W$ be the map from intercepts $\{i_b\}_{b\in B_0}$ to tile areas $\{A_w\}_{w\in W}$. It is a quadratic function,
that is, each area is a quadratic function of the intercepts.  
From our definition of intercepts, the differential $D\Psi$ satisfies $\partial A_w/\partial i_b = \pm|K_{w,b}|,$
that is, plus or minus the length of the edge of face $w$ in direction $b$. 
The sign here depends on which side of the edge the face is on.
This matrix $D\Psi = (\pm|K_{w,b}|)$ is in fact gauge equivalent to $K$: $K_{w,b} = \pm e^{i\theta_b}(D\Psi)_{w,b} $
where $\theta_b$ is the direction of edge $b$ and the sign depends only on $b$. 

Thus we see that $K_{\Psi} := D\Psi$ is a Kasteleyn matrix for $\G$ which is also 
the differential of the map $\Psi$. 

The inverse $K_{\Psi}^{-1}$ is the differential
of the map $\Psi^{-1}$ from areas to intercepts.
Let $\G_0$ be the
graph obtained from $\G$ by removing all black boundary vertices except $b_0$. Dimer covers
of $\G_0$ are counted by $|\det K_\Psi|$.
We can use $K_{\Psi}^{-1}$ to relate the tiling geometry to the edge probabilities of a random dimer covering of $\G_0$, as follows.
The probability of edge $wb$ is $K(w,b)K^{-1}(b,w)$ (and this quantity is independent of choice of gauge), see \cite{Kenyon.localstats}. The fact that the probability
is positive means that $\frac{\partial A_w}{\partial i_b}$ and $\frac{\partial i_b}{\partial A_w}$ have the same sign. This implies that if we increase the area of a tile $t$ by a small amount, keeping the other tile areas fixed
then each edge of $t$ moves \emph{outward}, as if we were blowing air into the tile. 
The probability $p_{wb}$ is the relative proportion of area along $b$ which swept out as $A_w$ increases
(and $i_b$ moves outwards).

For multiple edges we can also draw a geometric conclusion. Given two edges $w_1b_1$ and $w_2b_2$ of $\G$,
suppose the signs of the intercepts are chosen so that $K_{w_1,b_1}, K_{w_2,b_2}>0$.
The probability of the pair of edges $w_1b_1$ and $w_2b_2$ is (see \cite{Kenyon.localstats})
$$K(w_1,b_1)K(w_2,b_2)\det\begin{pmatrix}K^{-1}(b_1,w_1)&K^{-1}(b_2,w_1)\\K^{-1}(b_1,w_2)&K^{-1}(b_2,w_2)\end{pmatrix}.$$
The fact that this is positive implies that if we vary only the
areas $A_{w_1},A_{w_2}$, keeping other areas fixed, the map from those areas to the intercepts $i_{b_1},i_{b_2}$ is orientation-preserving. This is because the above determinant is the Jacobian of the map from areas $A_{w_1},A_{w_2}$ to intercepts $i_{w_1},i_{w_2}$. Consequently, for example, if $dA_{w_1}+\beta dA_{w_2}$ is a perturbation of areas which fixes $i_{b_1}$, and $\beta>0$, then it increases $i_{b_2}$ (and if $\beta<0$ it decreases $i_{b_2}$).

Here is another application.

\begin{theorem}\label{diffthm} $\Psi$ is a diffeomorphism from $C_T(\sigma)$ onto its image.
\end{theorem}

\begin{proof}
The differential of the map from intercepts to areas is the Kasteleyn matrix $K_{\Psi}$ and thus has nonzero 
determinant. Thus $\Psi$ is locally a diffeomorphism. Injectivity follows from Theorem \ref{areathm} above.
\end{proof}

Since $C_T(\sigma)$ is an open polytope, we have:
\begin{corollary}\label{ballcor}
For a family of tilings with fixed tile orientations $\sigma$, and whose underlying graph $\G$ has quad faces, 
the set of possible areas $\Psi(C_T(\sigma))$ is homeomorphic to a ball of dimension $d=(\text{\# tiles}-1)$.
\end{corollary} 

Without the hypothesis on quad faces the dimension of $C_T$ is smaller, 
so it is not clear what the corresponding statement should be.

\subsection{Homology tilings}\label{Homtilings}

For a convex tile $t$ recall that $V_t\cong\R^{n-2}$ is the vector space of closed polygonal paths 
with edges in the direction (up to sign) of the edges of $t$. We call such closed paths ``homology tiles''.
On $V_t$ let $q_t$ be the area form: the integral of $x\,dy$ over the closed path. We let $V_t^+,V_t^-$ be the two disjoint components of $V_t$ 
where $q_t>0$. These are interiors of cones; each $V_t^{\pm}$ contains one of the two components of $C_t$. 

Let $T$ be a convex tiling of a convex polygon $R$, whose underlying graph has quad faces. 
We can parameterize  $C_T=C_T(R)$ using the intercepts $i_b\in\R$.
In order for a tiling to exist the $i_b$ must satisfy certain inequalities.
If the $i_b$ fail to satisfy those inequalities we say we have a \emph{homology} tiling.
That is, the space of homology tilings is just the space $\R^{B_{int}}$ allowing
the intercepts to take arbitrary real values.

A homology tiling can be characterized in terms of winding numbers. For a point $p$ in general position inside $R$, the sum of winding
numbers of the homology tiles $t$ around $p$ is $1$, and for a point $p$ in general position outside $R$, the sum of winding numbers is zero.

We define $V_T\subset\R^{B_{int}}$ to be those homology tilings whose tiles have positive area.
On $\R^{B_{int}}$, a choice of orientation $\sigma$ for the tiles cuts out a convex region $V_T(\sigma)\subset V_T$, 
where individual homology tiles have fixed orientation. 

\begin{thm}\label{homologytilingthm} 
Let $T$ be a convex tiling of a convex polygon $R$, whose underlying graph has quad faces. 
Let $(A_1,\dots,A_{|W|})$ be a collection of positive reals summing to the area of $R$. Then there 
is a unique combinatorially equivalent homology tiling of $R$
such that tile $w$ has area $A_w$ and
the same tile orientation as the original tile. 
\end{thm}

\begin{proof}
$V_T(\sigma)$ is a convex set cut out by a finite number of cones. 
If $V_T(\sigma)$ is nonempty, it is bounded:
if it were unbounded then by going to $\infty$ along a line parallel to the boundary, some tile area would tend to $\infty$, contradicting the property
that the sum of the areas equals the area of $R$. 

For fixed positive $\{A_w\}$, the function $Q(\{i_b\}) = \sum_w A_w\log q_w$ is a
concave and analytic function on $V_T(\sigma)$, and tends to $-\infty$ on the boundary, and so has a unique critical point.
At the critical point
$$\sum_w\frac{A_w}{q_w}\frac{\partial q_w}{\partial i_b}=0,$$
that is 
$$\sum_w\frac{A_w}{q_w}K_{w,b}=0,$$
or
$$(\frac{A_w}{q_w})\in \ker K^*,$$
the left kernel of $K$ restricted to $B_{int}$. 
But recalling that  $\ker K^*$ contains $1_W$, if the kernel is of dimension $1$ then
$q_w=sA_w$ (for a global scale factor $s$) and we are done.

Generically in $\R^{B_{int}}$, $\ker K^*$ is of dimension $1$: this is true
for the original tiling and nearby, so true on a Zariski dense set.
So $\ker K^*$ is of dimension strictly larger than $1$ only on a hypersurface
$X$ defined by $\det K_{w_0}=0$, the determinant of $K$ when we remove row $w_0$
(since $1_W$ is in the left kernel of $K$, all maximal minors are equal,
so $X$ has no dependence on the choice of $w_0$.)
For generic areas $\ker K^*$ is thus one dimensional. The proof now follows by
compactness of $V_T$, taking a limit of solutions with generic areas.
\end{proof}

\begin{conjecture}\label{2^n} Let $T$ be a convex tiling of a convex polygon $R$, with the property that no two internal edges are parallel, and the underlying graph $\G$ has quad faces. 
Then for any choice of tile orientations $\sigma\in\{1,-1\}^W$ there exists a unique homology tiling of $R$ or $-R$
such that tile $w$ has positive area and orientation $\sigma_w$. 
\end{conjecture}

In Figure \ref{all8} see $8$ out of the $16$ possible solutions for a given set of areas (the other $8$ are obtained by rotation of these by $\pi$). 

\begin{figure}
\begin{center}\includegraphics[width=5.in]{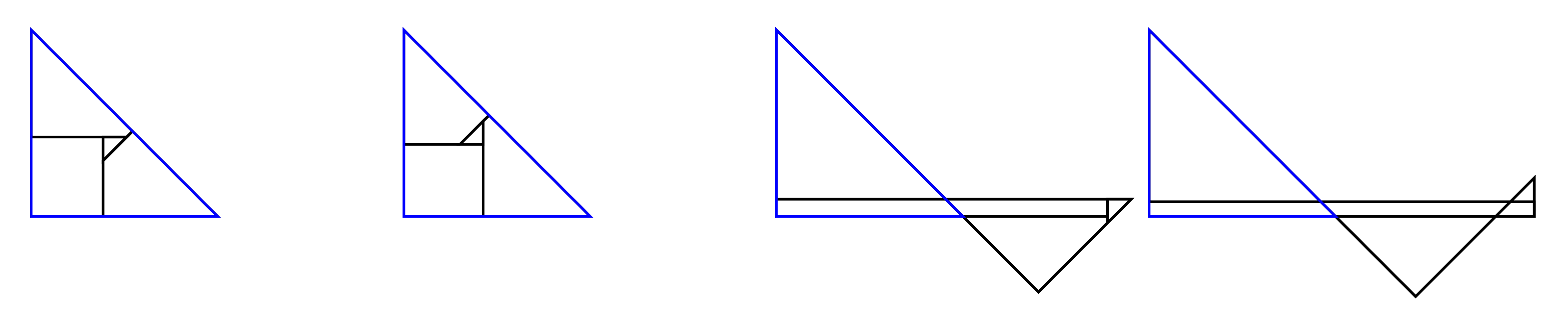}\\\includegraphics[width=5.in]{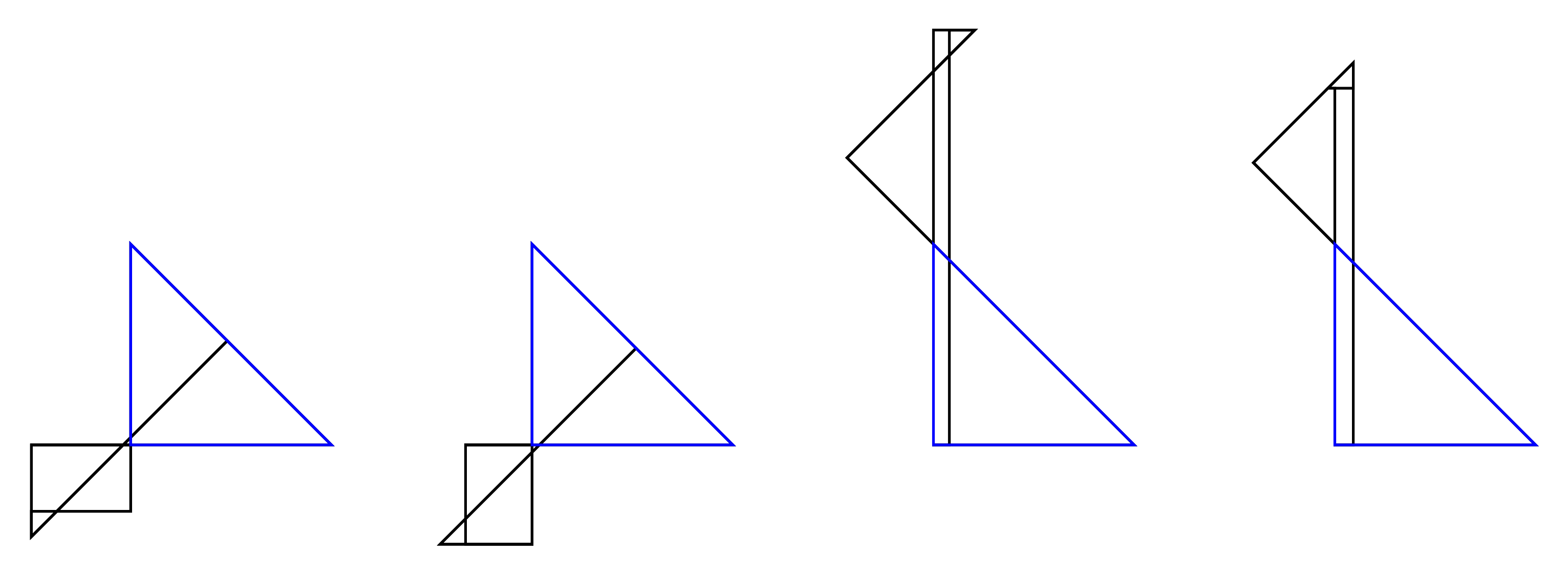}\end{center}
\caption{\label{all8}The $8$ homology tilings of a triangle $R$ (in blue) with areas equal to those of the first tiling.}
\end{figure}

If some edges are parallel, then apparently some of these solutions ``go to infinity" and we have fewer than $2^{W}$ tilings.
For example in the rectangle tiling case, we get significantly fewer, as discussed previously and in \cite{AK}. 

If all tiles are quadrilaterals,
$V_T$ is defined by a hyperplane arrangement; in this case bounded regions in the complement of the
hyperplanes correspond bijectively to orientations for which there is a homology tiling.

\end{document}